\documentclass[english,10pt]{paper}
\usepackage{amsmath,amsfonts,amssymb,amsbsy,amsthm}
\usepackage{subfigure}
\usepackage{graphicx}
\usepackage{wrapfig}
\usepackage[unicode=true, pdfusetitle,
 bookmarks=true,bookmarksnumbered=false,bookmarksopen=false,
 breaklinks=false,pdfborder={0 0 1},backref=false,colorlinks=false]
 {hyperref}

\makeatletter

\newtheorem{proposition}{Proposition}

\newtheorem{theorem}{Theorem}
\newtheorem{lemma}{Lemma}

\newtheorem{conjecture}{Conjecture}
\theoremstyle{definition}

\def\blfootnote{\xdef\@thefnmark{}\@footnotetext}

\makeatother

\begin{document}

\title{Lanchester Theory and the Fate of Armed Revolts}

\author{Michael P. Atkinson $^{1}$, Alexander Gutfraind $^{2}$, Moshe Kress $^{1}$}

\maketitle

\begin{abstract}
Major revolts have recently erupted in parts of the Middle East with substantial international repercussions. Predicting, coping with and winning those revolts have become a grave problem for many regimes and for world powers.  We propose a new model of such revolts that describes their evolution by building on the classic Lanchester theory of combat. The model accounts for the split in the population between those loyal to the regime and those favoring the rebels.  We show that, contrary to classical Lanchesterian insights regarding traditional force-on-force engagements, the outcome of a revolt is independent of the initial force sizes; it only depends on the fraction of the population supporting each side and their combat effectiveness.  We also consider the effects of foreign intervention and of shifting loyalties of the two populations during the conflict.  The model's predictions are consistent with the situations currently observed in Afghanistan, Libya and Syria (Spring 2011) and it offers tentative guidance on policy.
\end{abstract}

\section{Introduction}
\setcounter{footnote}{3} \blfootnote{$^{1}$Theoretical Division, Los Alamos National Laboratory, Los Alamos, New Mexico,\href{mailto:agutfraind.research@gmail.com}{agutfraind.research@gmail.com}. ~~~ $^{2}$Operations Research Department, Naval Postgraduate School, Monterey, California, \href{mailto:mpatkins@nps.edu}{mpatkins@nps.edu},\href{mailto:mkress@nps.edu}{mkress@nps.edu}.}%
\label{sec:intro} Recent (2011) events in Libya underscore the significant
impact of armed revolts on regional and global interests. 
Armed revolts typically start with demonstrations and civic unrest
that quickly turn into local violence and then full-scale combat. (The terms
revolt, rebellion, and insurgency are interchangeable in most senses and we
use the term revolt throughout for consistency.) As demonstrated in Libya 
the evolution of the armed revolt has a strong spatial component;
individuals in some regions (e.g., parts of Tripoli) may be loyal to the
regime because of ideology or economic and social incentives or fear, while
other regions (e.g., Benghazi) become bastions of the rebels powered by strong
local popular support. Thus, armed revolts, very much like conventional war,
are about gaining and controlling populated territory. However, unlike
conventional force-on-force engagements, where the civilian population plays
a background role, armed revolts are characterized by the active role of the
people, who become a major factor in determining the outcome of the
conflict: both the rebels and the regime need the support of the population
to 
carry out their campaigns~\cite{Hammes06,Kress09,Lynn05}. 

Our approach to modeling armed revolts is based on Lanchester theory~\cite%
{Lanchester16} that describes the strength of two opposing military forces
by two ordinary differential equations (ODEs). The forces cause mutual
attrition that depletes their strengths until one of the forces is defeated.
While Lanchester models are stylized and highly abstract, they have been
extensively used for analysis for almost a century because they provide
profound insights regarding conditions that affect the outcomes of military
conflicts~\cite{Washburn09}.

Using our model, we derive the end-state of the revolt, identify stalemate
situations and study the effects of foreign intervention and of inconstant 
support by the population. We show that contrary to classical Lanchesterian
insights regarding traditional force-on-force engagements, the outcome of a
revolt is independent of the initial force sizes. We also derive conditions
for successful foreign interventions. These results are consistent with the
situations currently (Spring 2011) observed in Afghanistan, Libya and Syria.
We also evaluate policy options facing the international community. 

\section{Setting and Assumptions}

\label{sec:settings} Consider an armed revolt involving two forces, termed
Red and Blue, that rely on the population for manpower, intelligence,
and most other resources. The population is divided into supporters of the
Blue, called henceforth \textit{supporters}, and supporters of Red, called
henceforth \textit{contrarians}. We initially assume that the support
strongly depends on factors such as tribal affiliation, social class, and
ideology and therefore remains unchanged during the armed revolt. However,
later on we relax this assumption and allow for changes in popular behavior,
reflecting pragmatic and opportunistic responses of the population to
changes in the force balance.

We assume that the country is divided between Red and Blue and therefore a
populated region lost by one force is gained by the other force. A force
that fights over a region might be either supported or opposed by the local
population, situations which we call \textit{liberation} or \textit{%
subjugation}, respectively. A liberating force
fights more effectively than a subjugating force because of population
support, \textit{ceteris paribus}. Moreover, the forces in control of hostile regions are busy
policing the population and therefore adopt a defensive posture. Thus,
only the forces operating in friendly regions proactively attempt to capture
additional territories.

\section{Model}

\label{sec:model}

Let $S$ and $C$  ($S+C=1$) denote the fraction of the total population who are
supporters of Blue and supporters of Red (\textquotedblleft
contrarians\textquotedblright ), respectively. Let also $B$ and $%
R$  ($B+R=1$) denote the fraction of the population controlled by Blue and Red,
respectively. We use the notation $XY$ for the fraction of
population $X$ that is controlled by force $Y$, where $X=S,C$ and $Y=B,R$.
Hence, $SB+SR=S$ and $CB+CR=C$. When Blue subjugates a $CR$ region it
becomes part of $CB$ and when Blue liberates an $SR$
area it becomes part of $SB$.
Similar actions are possible by Red, giving a total of four kinds of combat engagements, 
as shown in Figure~\ref{fig:model}.

Because Red and Blue operate in populated areas, the outcome of an
engagement depends both on the strength of the attacking force but also
on the signature (i.e. visibility) of the defending force; smaller attack
force (\textquotedblleft fewer shooters\textquotedblright ) or smaller
signature (\textquotedblleft fewer targets\textquotedblright ) result in a
smaller gain/loss rate. Namely, at each interaction, the gain rate of the
attacker is given by a scaling constant, called henceforth \textit{attrition
rate}, multiplied by the product of the attacking and defending force sizes.
This relationship implies that even a large attacker would struggle to
suppress a small defender diffused in the population. The resulting model is an
adaptation of the Lanchester Linear Law (see e.g.,~\cite{Washburn09} p.~83) 
and Deitchman's guerrilla warfare model~\cite{Deitchman62}.

\begin{figure}[tbp]
\centerline{\includegraphics[width=.45\textwidth]{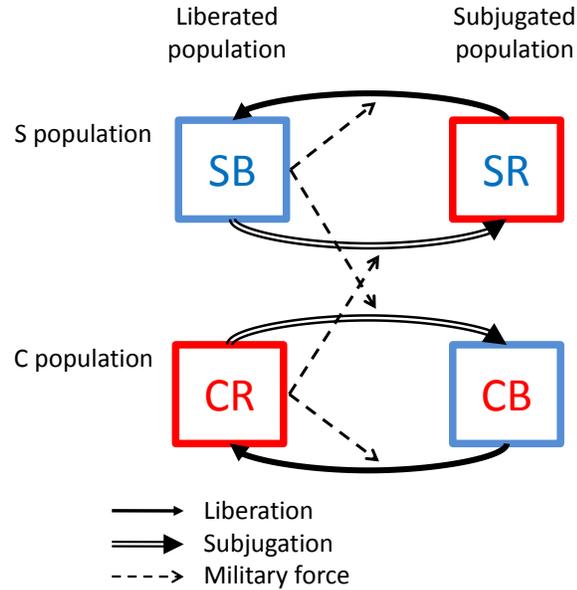}}
\caption{Schematic dynamics of the model. The four variables in the model appear as boxes,
where each box represents a possible combination of population behavior and controlling
force. Solid lines indicate change in control of population while dashed
lines indicate the force causing it. Observe that the population does not
change allegiances even under occupation.}
\label{fig:model}
\end{figure}

The attrition rate constants depend on the tactics and equipment of the parties 
as well as the behavior of the population.
Thus, let $f_{S}$ and $f_{C}$ denote the rates of liberation of friendly regions
by Blue and Red forces, respectively. Similarly, let $h_{C}$ and $h_{S}$ denote
the rates of subjugation of hostile regions by Blue and Red, respectively.
The resulting dynamics are given in Eqs.~\ref{eq:model}.
\begin{align}
SB^{\prime }& =+f_{S}SB\cdot SR-h_{S}CR\cdot SB  \notag \\
SR^{\prime }& =-f_{S}SB\cdot SR+h_{S}CR\cdot SB  \label{eq:model} \\
CR^{\prime }& =+f_{C}CR\cdot CB-h_{C}SB\cdot CR  \notag \\
CB^{\prime }& =-f_{C}CR\cdot CB+h_{C}SB\cdot CR  \notag
\end{align}%
Since it is easier to fight in friendly territory, we make the following 
\textit{dominance assumption:} 
\begin{equation}
f_{S}>h_{C}\mbox{ and }f_{C}>h_{S}\,.  \label{eq:dominance}
\end{equation}%

\section{End-State of the Revolt}

\label{ssec:model:vic} From solving Eqs.~\ref{eq:model} we obtain that the
conflict can result in one of three outcomes, corresponding to the stable
equilibrium points of the equations:

\begin{enumerate}
\item \emph{Blue victory:} $SB+CB=1$,

\item \emph{Red victory:} $CR+SR=1$,

\item \emph{Stalemate:} Both sides control a fraction of the total
population.
\end{enumerate}

It can be shown that the evolution of the conflict does not involve cycles
where populated regions change sides endlessly; rather, the conflict
dissipates and reaches a stable state 
(proofs of this and all other results are given in the Supporting
Information at the end of this article.) 

The stable outcomes are not dependent on all four attrition rates
but rather on two ratios: $r_{S}=\frac{f_{S}}{h_{S}}$ and $r_{C}=\frac{f_{C}}{h_{C}}$.
We call these the \textquotedblleft liberation-subjugation effectiveness
ratio\textquotedblright\ (LSER) of supporters and contrarians, respectively.
These ratios account for differences in tactics, technology, and information
between Blue and Red, and also reflect the ability and commitment
of the local population to support its preferred force. The outcomes are:
\begin{align}
\text{Blue wins if and only if }& \qquad r_{C}<\frac{S}{1-S}
\label{eq:gov-vic} \\
\text{Red wins if and only if }& \qquad r_{S}<\frac{1-S}{S}
\label{eq:ins-vic} \\
\text{Stalemate occurs otherwise.}  \notag
\end{align}
These results\footnote{%
Technically, we assume that at the start of the dynamics both forces have
some presence in a friendly territory, i.e. $SB_{0}>0$ and $CR_{0}>0$.
Otherwise, one of the forces is never challenged and wins trivially. Also,
the model has a fourth equilibrium that corresponds to the case where the
territory is divided between Blue and Red who control only hostile territory
($SR+CB=1$). Obviously, such a situation is very unlikely and indeed this
equilibrium is unstable, as shown in the Supporting Information.}
are summarized in Figure~\ref{fig:equils}(left). It follows
from Ineqs.~\ref{eq:gov-vic}--\ref{eq:ins-vic} that the fate of the armed
revolt is completely determined by the LSERs and the population split between
supporters $S$ and contrarians $C=1-S$; it does not depend on the initial
sizes of the Blue and Red forces. Moreover, the minimum popular support
needed to guarantee Blue's win only depends on the LSER in the contrarians'
territory. Specifically, Blue wins if and only if $r_{C}(1-S)<S,$ that is,
if the fraction of its supporters is larger than the fraction of contrarians
times the LSER in contrarians' territory. An equivalent statement applies for
Red victory, which happens if and only if $r_{S}(1-C)<C$. The operational
implication of these two conditions is that strengthening one's advantage in
friendly territories (e.g., Blue increasing $r_{S}$) may be sufficient to
avoid defeat but not to secure a win; if one is not effectively fighting in
hostile territory (e.g., Blue cannot sufficiently decrease $r_{C}$) then
one cannot win; the best it can hope for is a stalemate. At the
stalemate equilibrium, denoted $XY_{b}$  
\begin{align}
CB_{b}& =\frac{S(1+r_{S})-1}{r_{S}r_{C}-1}, & SR_{b} &=\frac{r_{C}-S(1+r_{C})}{%
r_{S}r_{C}-1},  \label{eq:equil} \\
SB_{b}& =r_{C}CB_{b}, & CR_{b} & =r_{S}SR_{b}.  \notag
\end{align}%
Notice that the denominators are always positive because of the dominance
assumption (see Ineqs.~\ref{eq:dominance}). Eqs.~\ref{eq:equil} indicate that as $S$ increases an increasing part of the population is controlled by Blue.
When $r_{C}$ increases, a larger fraction of the contrarians is able to remain
free (i.e. ruled by Red). 

\begin{figure}[tbp]
\centerline{\includegraphics[width=.4\textwidth]{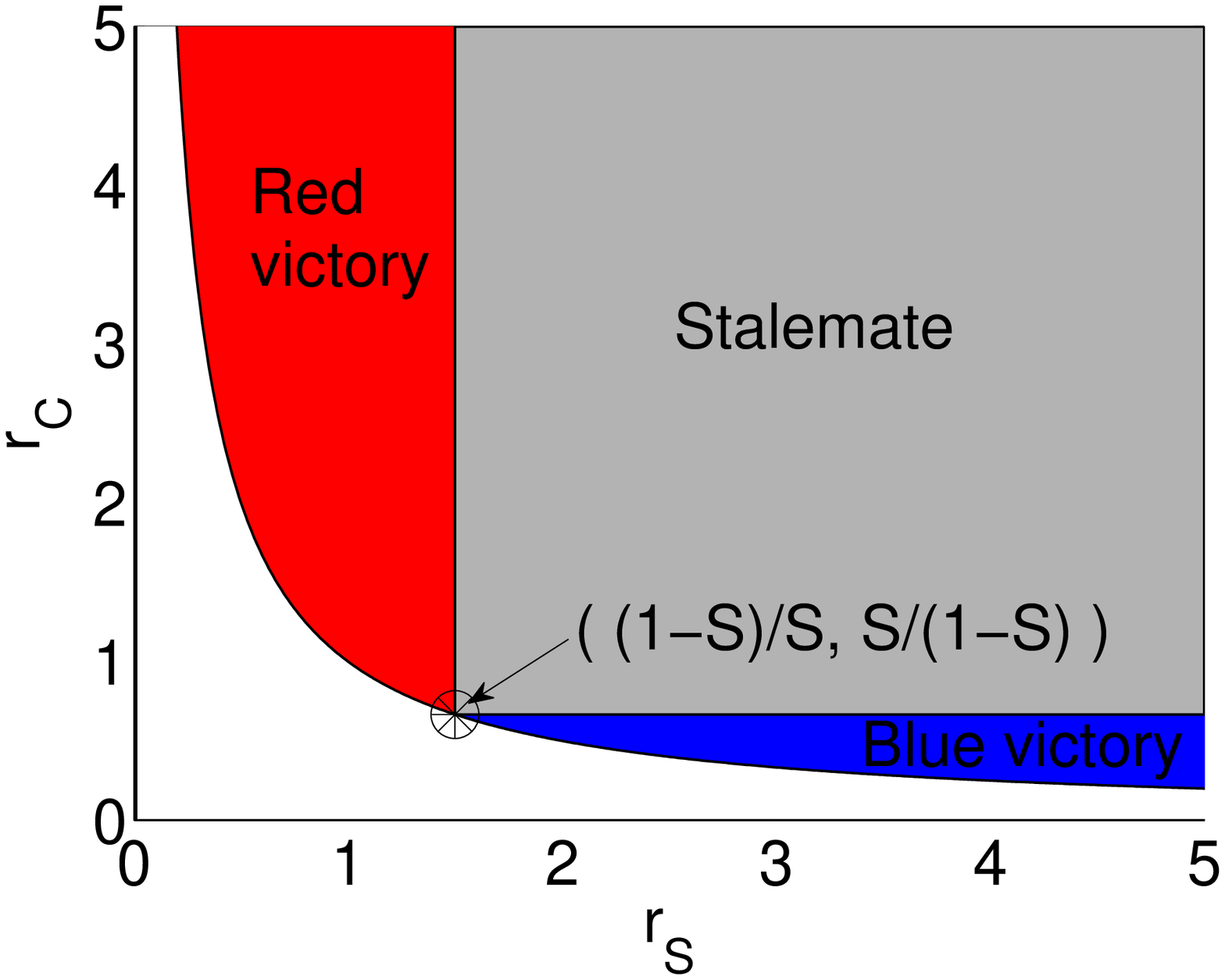}%
\includegraphics[width=.4\textwidth]{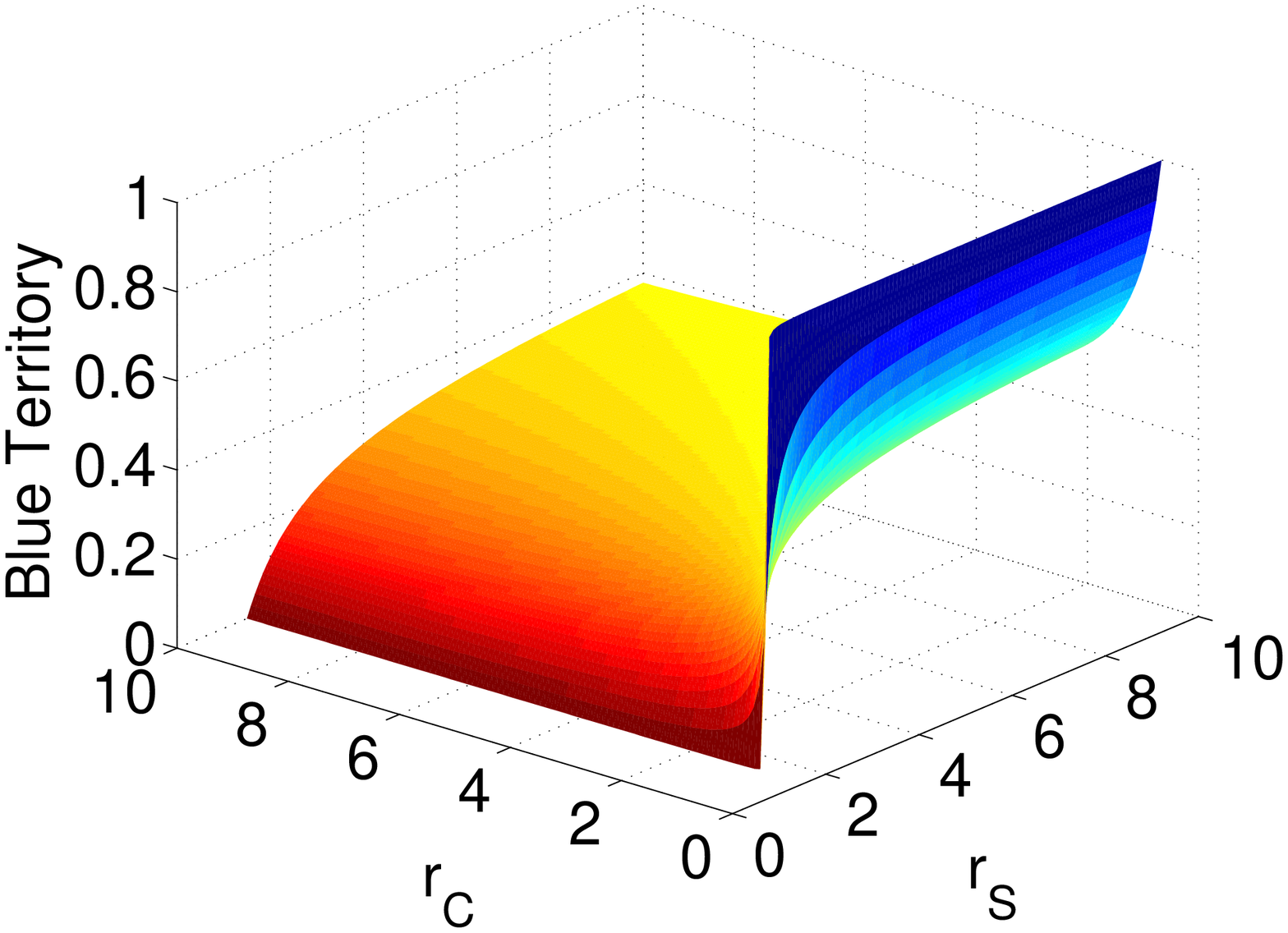}}
\caption{Outcomes of the conflict when $S=40\%$ as a function of $r_{S}$ and 
$r_{C}$. (left) The possible outcomes are: stalemate (gray), Red victory
(red) and Blue victory (blue). The white area is excluded by the
dominance assumption (Ineqs.~\protect\ref{eq:dominance}). Outright victory
is possible only when one party has a low LSER. Increasing $r_{S}$ and $%
r_{C} $ makes both parties much more entrenched in their areas, leading to a
stalemate regardless of the value of $S$. (right) The amount of territory
controlled by Blue. Observe that a very sharp change in the outcome is
predicted as $r_C$ approaches $\frac{2}{3}$, from a balanced stalemate to a
Blue victory.}
\label{fig:equils}
\end{figure}

We plot the fraction of the population controlled
by Blue during a stalemate (i.e., $CB+SB$) in Figure~\ref{fig:equils}(right). 
(We only present the plot for $S=0.4$, but other values of $S$ are qualitatively similar.)
Near the Blue victory condition defined by Ineq.~\ref%
{eq:gov-vic} the fraction of the population controlled by Blue is near one, but
quickly decreases as $r_{C}$ increases. Similarly, the fraction of
population controlled by Blue rapidly increases as $r_{S}$ moves away from
the Red victory condition. However, after the significant initial change in
the fraction of Blue's regions as $r_{S}$ or $r_{C}$ increases, the surface
flattens out. As both $r_{S}$ and $r_{C}$ continue to increase, the fraction
of Blue's regions approaches $S$. Therefore, when $r_{S}$ and $r_{C}$ are
reasonably bounded away from their thresholds, an entrenched stalemate
occurs where Red and Blue control primarily their friendly territories.

\section{Extensions of the Basic model}

\label{sec:xtend} We consider now two extensions of the basic model: the
case of foreign intervention and the case of shifting popular support.

\subsection{Foreign Military Intervention}

\label{ssec:model:fa} Most large revolts in modern times involved
foreign military interventions by regional or
global powers~\cite{Sarkes10,Small82}. Such interventions can be manifested
in two ways: \textit{direct} and \textit{indirect}. Direct intervention
(e.g., air-strike support to ground units, such as the intervention of NATO
forces in Libya in 2011) allows the supported side to exercise more
firepower against its opponent. Indirect intervention provides the supported
side with force multipliers such as intelligence, training, logistical
support and advanced weapons, but no additional firepower per se. 
In both cases we assume that just one side, say Blue, receives the foreign support.
We leave for future studies to consider the case of foreign support to both sides.

\paragraph{Direct intervention.}
For simplicity, suppose that the foreign constituent is tactically superior and it
experiences negligible attrition (e.g., air support for Blue that is subject
to ineffective air defense of Red). 
Therefore, the effectiveness of the
foreign constituent remains fixed throughout the armed revolt. However,
similarly to the direct engagements discussed above, its ability to target
Red diminishes as the size of Red's forces decreases. In that case, Red
targets are harder to find and engage. Let $\lambda _{S},\lambda _{C}>0$
denote the combat power of the foreign constituent when operating in
supporters' $(S)$ regions and contrarians' $(C)$ regions, respectively. The
separation into two combat power parameters allow for the possibility that
the foreign constituent only contributes to certain kinds of operations
(e.g. only to liberating supporters), and/or is affected by the behavior of
the population, just like Blue. In this case Eqs.~\ref{eq:model} become 
\begin{align}
SB^{\prime }& =+f_{S}SB\cdot SR-h_{S}CR\cdot SB+\lambda _{S}SR  \notag \\
SR^{\prime }& =-f_{S}SB\cdot SR+h_{S}CR\cdot SB-\lambda _{S}SR
\label{eq:model-fa} \\
CR^{\prime }& =+f_{C}CR\cdot CB-h_{C}SB\cdot CR-\lambda _{C}CR  \notag \\
CB^{\prime }& =-f_{C}CR\cdot CB+h_{C}SB\cdot CR+\lambda _{C}CR  \notag
\end{align}%
Since the effectiveness of the foreign constituent remains unaffected, it is
clear that Red cannot win. The only two outcomes are Blue's victory and a stalemate. Blue wins if and only if $\lambda
_{C}>f_{C}(1-S)-h_{C}S.$ Otherwise, the armed revolt ends in a stalemate.
Like in the basic model, the conflict dissipates and reaches a stable state,
and no cycles are possible.

An interesting observation is that the value of $\lambda _{S}$ -- the combat
power of the foreign constituent in friendly regions -- plays no role in
helping Blue achieve victory; it only ensures that Blue will not lose as
long as $\lambda _{S}>0$. 
The threshold of $\lambda _{C}$ that determines
Blue's victory is the difference between two terms, each a combination of
combat effectiveness and popular support: $f_{C}(1-S)$ is Red's
effectiveness fighting on friendly territory times its popular support, and $%
h_{C}S$ is Blue's effectiveness fighting on hostile territory times its
popular support. Clearly, {this threshold decreases as the support to Blue
increases.} In particular, a sufficient condition for Blue victory is $%
\lambda _{C}>f_{C}$, which only depends upon the fighting effectiveness of
Red. Consequently, even if Blue has limited tactical capabilities or a small
amount of popular support, it can still prevail with enough assistance from
a foreign constituent.

\paragraph{Indirect intervention.}

Indirect intervention (force multiplier) increases the ability of Blue to
defend its territory and to attack Red forces. Specifically, the liberation
rate $f_{S}$ and the subjugation rate $h_{C}$ are multiplied by factors $%
\mu _{S},\mu _{C}>1$, respectively, where the structure of Eqs.~\ref%
{eq:model} remains unchanged. The LSER values $r_{S}$ and $r_{C}$ change to $%
\mu _{S}r_{S}$ and $\frac{r_{C}}{\mu _{C}}$, respectively. Using the
conditions in Ineq.~\ref{eq:ins-vic} we obtain that for Blue to avoid defeat
it is sufficient that the intervention be such that: 
\begin{equation*}
\mu _{S}\geq \frac{1-S}{r_{S}S}.
\end{equation*}%
In order to secure a win, it follows from Ineq.~\ref{eq:gov-vic} that the
support for Blue must be such that 
\begin{equation*}
\mu _{C}>\frac{r_{C}(1-S)}{S}.
\end{equation*}%
Because $r_{S}r_{C}>1$, the threshold of $\mu _{C}$ is always
larger than the threshold of $\mu _{S}$ -- it is more costly to secure a
victory than to avoid a loss. Obviously, the indirect intervention is needed to
secure a victory only if $S$ is small enough, specifically, if $S<\frac{r_{C}%
}{1+r_{C}}.$ Note that ``small enough'' may actually be quite large when Red
is very effective on its own turf compared with Blue ($r_{C}$ is large).

\subsection{Opportunistic Population}

\label{sec:model:dynamics} While in some conflicts the behavior of the
people is highly polarized and unchanging, in others the population might be quite
opportunistic and favor the side that appears more likely to win. It follows
that the fraction of the supporters, and hence contrarians, changes
according to the state of the conflict. We capture this situation by
treating the fraction of supporters $S$ as a dynamic variable, and adding to
Eqs.~\ref{eq:model} an equation for $S^{\prime}$. The value of $S^{\prime}$
increases with the fraction of population Blue controls ($SB+CB$) and
decreases with the fraction controlled by Red ($CR+SR$). Because $C=1-S$, $%
SR=S-SB$, and $CB=1-S-CR$, we obtain from Eqs.~\ref{eq:model} the three
independent equations: 
\begin{align}
SB^{\prime }& =+f_{S}SB(S-SB)-h_{S}CR\cdot SB  \notag \\
CR^{\prime }& =+f_{C}CR(1-S-CR)-h_{C}SB\cdot CR  \label{eq:model-dynamics} \\
S^{\prime }& =+\alpha (SB+1-S-CR)(1-S)-\alpha (CR+S-SB)S,  \notag
\end{align}%
where $\alpha $ is a parameter that determines the rate at which individuals
switch allegiances, which is assumed to be the same for both the supporters and contrarians. 
With opportunistic population there are only two potential outcomes:

\begin{enumerate}
\item \emph{Blue victory} where the entire population supports Blue, who
controls all regions ($SB=1$), and

\item \emph{Red victory} where the entire population supports Red, who
controls all regions ($CR=1$).
\end{enumerate}

These two equilibria are stable for all parameter values. There are also two
stalemate equilibria: a balanced stalemate where $SB,CR>0$, and a disarmed
stalemate where $SB=CR=0$. Neither of the two stalemate equilibria are
stable. The disarmed stalemate is neither realistic nor relevant. The
balanced stalemate, given below, is more interesting because it lies on a
boundary that separates the basins of attraction for the two victory
situations: 
\begin{align}
SB^{\ast }& =\frac{r_{C}}{2+r_{S}+r_{C}},  \label{eq:balstale:SB} \\
CR^{\ast }& =\frac{r_{S}}{2+r_{S}+r_{C}},  \label{eq:balstale:CR} \\
S^{\ast }& =\frac{1+r_{C}}{2+r_{S}+r_{C}}.  \label{eq:balstale:S}
\end{align}%
Thus, the stalemate equilibrium gives a rough metric for the potential
outcome of the conflict. For example, the closer the stalemate equilibrium is to the Blue
victory point, the more likely Red will win the conflict. This occurs
because most of the 3-dimensional $(r_{S},r_{C},S)$ parameter space lies in
the basin of attraction corresponding to Red victory.

\section{Discussion of Some Armed Revolts}

\label{sec:recentrevolts} Using the results above, we analyze several
ongoing (in 2011) revolts. 
While we could not access reliable data to empirically validate our
model (e.g., estimates of LSERs $r_{S}$ and $r_{C}$ or time-series data of conflicts), 
we show that many of the ongoing conflicts are at a state consistent with our model. 
The model suggests how their outcomes might be effected by decisions including those currently on the policy table.

\subsection*{Libya}

~~~ The available information regarding the revolt in Libya is based mainly
on fragmented news reports. As of June 2011, it appears that at least three
of the seven largest districts in Libya -- Benghazi, Misrata and Az-Zawiya
-- support the rebels, which amount to at least $40\%$ of the population. 
The Qaddafi forces
(labeled Red) are much better trained, equipped and organized than the rebels (labeled Blue).
Therefore, if $S$ denotes the supporters of the rebels, it is reasonable
to assume that $f_{S}<h_{S}$ and $f_{C}>h_{C}$ -- the regime forces have the
upper hand even in hostile regions, which implies that $r_{S}<1$ and $%
r_{C}>1.$ Ineq.~\ref{eq:ins-vic} suggests that Qaddafi should have achieved
a clear victory, crushing the revolt. Indeed, he appeared to be nearing
victory until the intervention of NATO forces that implemented a
no-fly zone and later launched a series of air-strikes against Qaddafi's
forces. In our terminology, NATO forces have applied direct intervention. 
This intervention, if sustained, implies that the rebels now cannot be defeated. 
From news reports it appears that the rebels have also been provided 
training and gear to help them repel Qaddafi's attempts 
at recapturing rebelling population, i.e. increasing the rebels' $r_{S}$.
At the same time, the rebels have been unable to make inroads into
territory currently controlled by Qaddafi, so $r_{C}$ is largely unchanged.
Under those conditions the model suggests that the most likely outcome is 
a protracted stalemate. 
To achieve victory for the rebels, the foreign intervention must focus on
decreasing $r_{C}$ and/or increasing $\lambda_{C}$, that is, attack or
help attack the regions supporting the regime.

\subsection*{Afghanistan}

~~~ One can view the ongoing conflict in Afghanistan (2001-) as a struggle
of government and coalition forces (Blue) against Salafists (Red) led by the
Taliban. Many observers of the conflict point to the critical need of both
parties to win the support of the population, so the conflict is a good
application of our model. According to the 2007 report by the International
Council on Security and Development, the Taliban have permanent presence in
54\% of the country~\cite{afghbrink2007}. Suppose, pessimistically, that the
Taliban movement has the support of all the people in the regions where it
is present. Assuming fixed behavior of the population (no opportunistic
shifts) the situation in Afghanistan will continue in its current stalemate
form unless $r_{S}<1.17$ (giving Red a victory) or $r_{C}<0.85$ (giving Blue
a victory). Thus, the model suggests that the
government can avoid a Taliban takeover of the country by nurturing the
support of the population it currently controls, and it is not necessary to
push back the Taliban from their areas.

\subsection*{Syria}

~~~ The civil unrest in Syria currently pits the Syrian army and
paramilitaries against massive but unarmed demonstrations. If the situation
escalates into an armed revolt, what might be its outcome? It would
appear that special units of the Syrian army possess superior tactics and
weapons sufficient to suppress any domestic military opponent. However, we
anticipate the regime to face a strong challenge because of its narrow base
of support in the Alawite sect. If we assume that loyalties do not shift,
then the rebels (Red) would be victorious against the government
(Blue) if $r_{S}<\frac{1-S}{S}$ (Ineq.~\ref{eq:ins-vic}). Assuming $S\approx
10\%$ (the entire Alawite community~\cite{CIAWFB}), 
the government must have LSER of $r_{S}\geq 9$ to avoid defeat. This
seems unlikely given the discontent even within the Alawite community and
the ease with which foreign backers of the rebels could provide them with
military hardware, such as armor-piercing munitions and air cover to
neutralize government forces. The current (June 2011) observed stalemate is
a result of small foreign intervention by Iranian and Hezbollah combatants.
This intervention is covert and therefore very fragile. Thus, we expect
that in an armed revolt along the lines of our model, Syria's Assad regime would be defeated. 

\section{Conclusions and Policy Implications}

We present a new Lanchester-type model that represents the dynamics of
liberating and subjugating populated regions in the setting of an armed
revolt. We identify winning and stalemate conditions and obtain some general
insights regarding the revolt's end state. Many revolts do not have a
decisive outcome, with both sides entrenched in a prolonged stalemate. Our
model explicitly identifies this realistic outcome, which is not captured in
classical Lanchester theory. Our model also illustrates that it is not
sufficient to ably control friendly regions; for victory it is crucial to be
able to effectively fight in hostile regions. 

We also study the effect of foreign intervention (e.g., NATO intervention in
Libya) on the outcome of a revolt. We find that while direct intervention to
support one side will prevent defeat of that side and can facilitate a win
even if that side has very little popular support, indirect intervention
cannot guarantee this. If the opponent (Red) has strong popular support and
the supported force (Blue) has a poor LSER such that $\frac{1-S}{r_{S}S}>1$,
then if $1<\mu _{S}<\frac{1-S}{r_{S}S},$ the supported force will actually
lose -- the indirect support is simply not large enough. The level of
foreign intervention (either direct or indirect) required to defeat an
opponent depends on the popular support $(S)$ and the attrition coefficients
($f_{C}$ and $h_{C}$) in the contrarians' territory; it does not depend on
the capabilities of the forces in supporters' regions.

Finally if the population can shift its support, then a stalemate is not
possible. A bandwagon-type effect will occur where the population increases
its support to the apparent winner, which strengthens it and leads to more
support, which further strengthens it and so on until the side achieves victory. 
Unlike the case of fixed population behavior, the results of this scenario are sensitive to the initial conditions.

\section*{Acknowledgments}
Part of this work was funded by the Department of Energy at the Los Alamos
National Laboratory through the LDRD program, and by the Defense Threat
Reduction Agency (to AG). AG thanks Aric Hagberg and Feng Pan.
Indexed as Los Alamos Unclassified Report LA-UR-11-03477.

\section{Supporting Information}

\label{sec:proofs} We present here the mathematical proofs of the results
presented in the main text. 

\subsection{Victory Conditions}

\label{sec:proofs:vic} 
In general, dynamical systems may exhibit stable oscillations. We now show
that the system of equations in Eqs.~\ref{eq:model} does not have those
oscillations (no limit cycles). This means that the state variables will
always reach one of the equilibrium points (by the Poincar\'{e}-Bendixon
Theorem~\cite{Strogatz00}).

\begin{theorem}[Dulac's Criterion \protect\cite{Strogatz00}]
Let $\dot{x}=f(x)$ be a smooth system on a simply-connected set $S\subset%
\mathbb{R}^{2}.$ Let $w:S\to\mathbb{R}$ be smooth on $S$. Suppose on $S$ the
expression $\nabla\cdot\left(f(x)w(x)\right)$ does not change sign. Then the
system has no limit cycle on $S$.
\end{theorem}

\begin{proposition}
\label{prop:nocycle-basic} The system of Eqs.~\ref{eq:model} does not have
limit cycles.
\end{proposition}

\begin{proof}
We first write the model as a system of two independent equations by removing
the mixed variables:
\begin{align}
SB^{\prime } & = +f_{S}SB(S-SB)-h_{S}CR\cdot SB \label{eq:2varSB}\\
CR^{\prime } & = +f_{C}CR(C-CR)-h_{C}SB\cdot CR \label{eq:2varCR}
\end{align}

\noindent Let \[
w(SB,CR)=\frac{1}{SB(S-SB) CR(C-CR)}.\]
Both the system and $w$ are smooth on the set $(0,S)\times(0,C)$ (points on
the boundary of this space move to one of the fixed points and do not oscillate.)

\begin{align*}
\nabla\cdot\left(fw\right) & = \nabla\cdot\left(\frac{f_{S}(S-SB)-h_SCR}{(S-SB) CR(C-CR)},\frac{f_C(C-CR)-h_CSB}{SB(S-SB)(C-CR)}\right)\\
 & = \frac{-h_S}{(S-SB)^{2}(C-CR)}+\frac{-h_C}{(S-SB)(C-CR)^{2}} < 0.
\end{align*}
\end{proof}

We next show that the victory conditions in \ref{eq:gov-vic}, \ref%
{eq:ins-vic} correspond to the stability conditions for the equilibria
points. Throughout, we assume that $0<S<1$, i.e. $S$ is not on its boundary.

\begin{theorem}
\label{thm-vic-con} The following four statements hold for the system of
differential equations defined by Eqs.~\ref{eq:model}:

\begin{itemize}
\item The equilibrium $CR = SB = 0$ and $SR = 1 - CB = S$ is never stable.

\item The Blue victory equilibrium ($CR = SR = 0$ and $SB = 1 - CB = S$) is
stable if and only if $r_{C}<\frac{S}{1-S}$.

\item The Red victory equilibrium ($SB = CB = 0$ and $CR = 1 - SB = 1-S$)
is stable if and only if $r_{S}<\frac{1-S}{S}$.

\item The stalemate equilibrium (defined by Eqs.~\ref{eq:equil})) is stable
if and only if $r_{C}\geq \frac{S}{1-S}$ and $r_{S}\geq \frac{1-S}{S}$.
\end{itemize}
\end{theorem}

\begin{proof}
The model is fully specified based on two variables: $SB$ and $CR$, Eqs.~\ref{eq:2varSB}-\ref{eq:2varCR}. We first compute the Jacobian of the right hand size of differential equation
\begin{align*}
\mathbf{J}(SB,CR) = 
\begin{pmatrix}
f_{S}(S-2SB)-h_S CR & -h_SSB\\
-h_C CR & f_C(1-S-2CR)-h_CSB
\end{pmatrix}
\end{align*}
A solution $(SB^*,CR^*)$ to the differential equation is stable if the two eigenvalues of $\mathbf{J}(SB^*,CR^*)$ have negative real parts~\cite{Strogatz00}.  
By inspection the equilibrium with $SB=CR=0$ is not stable for any parameter values. 

\noindent {\bf Blue Victory}\\
The characteristic polynomial in this case is 
\begin{align*}
(-f_{S}S - \lambda)(f_C(1-S)-h_CS - \lambda)=0\,.
\end{align*}
The first eigenvalue, $-f_{S}S$, is always negative and the second  eigenvalue, $f_C(1-S)-h_CS$, is negative if $r_{C}<\frac{S}{1-S}$.

\noindent {\bf Red Victory}\\
The characteristic polynomial in this case is 
\begin{align*}
(-f_C(1-S) - \lambda)(f_{S}S-h_S (1-S) - \lambda)=0\,.
\end{align*}
The first eigenvalue, $-f_C(1-S)$, is always negative and the second eigenvalue, 
$f_{S}S-h_S (1-S)$, is negative if $r_{S}<\frac{1-S}{S}$.  
By   the dominance assumption (i.e., Ineq.~\ref{eq:dominance}))  $r_Sr_C>1$ and  thus $\frac{1}{1+r_S}<\frac{r_C}{1+r_C}$.  
Therefore it is not possible for both the Blue victory and Red victory to be stable equilibria for the same values of $r_S$ and $r_C$.  
This also implies that if  
$r_{S}<\frac{1-S}{S}$ a Blue victory cannot occur and if $r_{C}<\frac{S}{1-S}$ a Red victory cannot occur. 

\noindent {\bf Stalemate}\\
At the stalemate equilibrium in section~\ref{ssec:model:vic}, the variables have the values $SB_b$, $SR_b$, $CR_b$, and $CB_b$.
We next present a lemma, and then prove that the stalemate equilibrium is stable if and only if $\frac{1}{1+r_S} < S < \frac{r_C}{1+r_C} $, which is equivalent to the two conditions $ r_{C}\geq \frac{S}{1-S}$ and $r_{S}\geq \frac{1-S}{S}$
\begin{lemma}
\label{lem:1}
The product $SB_b CR_b$ is positive if and only if $\frac{1}{1+r_S} < S < \frac{r_C}{1+r_C}$.
\end{lemma}
\begin{proof}
By Eqs.~\ref{eq:equil}, $SB_b$ is positive if $\frac{1}{1+r_S} < S $, and $CR_b$ is positive if $S < \frac{r_C}{1+r_C}$.     By the dominance assumption (Ineq.~\ref{eq:dominance}) $\frac{1}{1+r_S}<\frac{r_C}{1+r_C}$, and therefore it is impossible for  $SB_b$ and  $CR_b$ to both be negative.
\end{proof}
Lemma~\ref{lem:1} and Eqs.~\ref{eq:equil} imply that if $\frac{1}{1+r_S} < S < \frac{r_C}{1+r_C}$, then $SB_b$, $SR_b$, $CR_b$, and $CB_b$ are all positive (and by conservation of total population, less than 1).  The Jacobian matrix for the stalemate equilibrium is 
\begin{align*}
\mathbf{J}(SB_b,CR_b) = 
\begin{pmatrix}
-f_{S}SB_b & -h_SSB_b\\
-h_C CR_b & -f_CCR_b
\end{pmatrix}
\end{align*}
We derive the upper left hand element of this Jacobian below
\begin{align*}
\mathbf{J}_{11}(SB_b,CR_b) &= f_{S}(S-2SB_b)-h_SCR_b\\
&= h_S(r_S (S-2SB_b)-r_SSR_b)\\
&= f_{S}\left(S-\frac{2Sr_C(1+r_S)-2r_C}{r_Sr_C-1}-\frac{r_C-S(1+r_C)}{r_Sr_C-1}\right)\\
&= -f_{S}SB_b
\end{align*}
The lower right hand element of $\mathbf{J}(SB_b,CR_b)$ can be derived in a similar fashion and we omit the details.  
The eigenvalues of $\mathbf{J}(SB_b,CR_b)$ will both have a negative real component if the trace of $\mathbf{J}(SB_b,CR_b)$ is negative and the determinant is positive.
The determinant of $\mathbf{J}(SB_b,CR_b)$  is $ SB_b CR_b h_Ch_S(r_Sr_C-1)$.  
If $\frac{1}{1+r_S} < S < \frac{r_C}{1+r_C}$, then by Lemma~\ref{lem:1} the trace is negative 
and the determinant is positive and thus the stalemate equilibrium is stable. 
If $S \notin \left(\frac{1}{1+r_S}, \frac{r_C}{1+r_C}\right)$ then by Lemma~\ref{lem:1}
the determinant of $\mathbf{J}(SB_b,CR_b)$  is negative and thus one of the eigenvalues has a positive real component and the stalemate equilibrium is not stable.
\end{proof}

\subsection{Direct Foreign Intervention}

\label{sec:proofs:fa} Let us rewrite the direct intervention dynamics of
Eqs.~\ref{eq:model-fa}: 
\begin{align}
SB^{\prime }& =+f_{S} \left(SB + \frac{\lambda_S}{f_S}\right) SR-h_{S}
CR\cdot SB  \notag \\
SR^{\prime }& =-f_{S} \left(SB + \frac{\lambda_S}{f_S}\right) SR+h_{S}
CR\cdot SB  \label{eq:model-directint} \\
CR^{\prime }& =+f_{C} CR\cdot CB-h_{C} \left(SB + \frac{\lambda_C}{h_C}%
\right) CR  \notag \\
CB^{\prime }& =-f_{C} CR\cdot CB+h_{C} \left(SB + \frac{\lambda_C}{h_C}%
\right) CR  \notag
\end{align}%
Furthermore let us define $A_S \equiv \frac{\lambda_S}{f_S}$ and $A_C \equiv 
\frac{\lambda_C}{h_C}$. 
The stalemate equilibrium of these equations is denoted with a
subscript $fi$ (``foreign intervention'') and we present them below. 
\begin{align}  
SB_{fi} & = \frac{SB_b}{2} - \frac{A_S}{2} + \frac{A_C-A_S}{2(r_Sr_C-1)} & + 
\sqrt{\left( \frac{SB_b}{2} - \frac{A_S}{2} + \frac{A_C-A_S}{2(r_Sr_C-1)}%
\right)^2 + \frac{r_Sr_CA_SS}{r_Sr_C-1} }  \label{eq:SG-fa} \\[10pt]
SR_{fi} & =S - SB_{fi} & \label{eq:SR-fa}  \\[10pt]
CB_{fi} & = \frac{SB_{fi} + A_C}{r_C} & \label{eq:CB-fa}  \\[10pt]
CR_{fi} & = 1-S - CB_{fi} & \label{eq:CR-fa} 
\end{align}
\noindent ($SB_b$ is the value at stalemate of the variable $SB$ in the
basic model, Eqs.~\ref{eq:equil}.) To derive these expressions we first
write the analog of Eqs.~\ref{eq:2varSB}--\ref{eq:2varCR}: 
\begin{align}
SB^{\prime } & = +f_{S}(SB+A_S)(S-SB)-h_SCR\cdot SB  \label{eq:2varSB-fa} \\
CR^{\prime } & = +f_CCR(1-S-CR)-h_C(SB+A_C) CR  \label{eq:2varCR-fa}
\end{align}
Solving Eqs.~\ref{eq:2varSB-fa}--\ref{eq:2varCR-fa} for the stalemate
equilibrium root results in two equations 
\begin{align}
f_{S}(SB+A_S)(S-SB) & = h_SCR\cdot SB  \label{eq:stale-faI} \\
f_C(1-S-CR)& = h_C(SB+A_C)  \label{eq:stale-faII}
\end{align}
Solving Eq.~\ref{eq:stale-faI} for $CR$ and substituting into Eq.~\ref%
{eq:stale-faII} produces a quadratic in $SB$. Solving for the positive root
of that quadratic yields the expression for $SB_{fi}$ in Eq.~\ref{eq:SG-fa}.
Substituting $SB_{fi}$ into Eq.~\ref{eq:stale-faII} gives $CB_{fi}$ in Eq.~%
\ref{eq:CB-fa}. \noindent Similarly to the basic model, it is possible to
exclude cycles, as follows.

\begin{proposition}
\label{prop:nocycle-intervention} The set of Eqs.~\ref{eq:model-fa} does not
have limit cycles.
\end{proposition}

\begin{proof}
We will work with the two independent Eqs.~\ref{eq:2varSB-fa} and~\ref{eq:2varCR-fa}.

\noindent Let \[
w(SB,CR)=\frac{1}{(SB+A_S) CR(1-S-CR)}.\]
Both the system and $w$ are smooth on the set $(0,S)\times(0,1-S)$. 

\begin{align*}
w SB^{\prime } & = \frac{f_{S}(SB+A_S)(S-SB)-h_S CR\cdot SB}{(SB+A_S) CR(1-S-CR)} \\
               & = \frac{f_{S}(S-SB)}{CR(1-S-CR)} - \frac{h_S}{(1-S-CR)}\left(1+\frac{A_S}{SB}\right)^{-1}\\
w CR^{\prime } & = \frac{f_C CR(1-S-CR)-h_C(SB+A_C) CR}{(SB+A_S) CR(1-S-CR)} \\
               & = \frac{f_C}{SB+A_S} - \frac{h_C(SB+A_C)}{(SB+A_S)(1-S-CR)}.
\end{align*}

\noindent Note that $1-S-CR = C-CR > 0$ in the strictly positive quadrant.  Therefore, 
\begin{align*}
\nabla\cdot\left(fw\right) & = \frac{-f_{S}}{CR(1-S-CR)} - \frac{h_S}{(1-S-CR)}(-1)\left(1+\frac{A_S}{SB}\right)^{-2}\frac{-A_S}{{SB}^2}\\
                           & + 0  - \frac{h_C(SB+A_C)}{(SB+A_S)(1-S-CR)^2} < 0.
\end{align*}
\end{proof}

Before proceeding to examine the stability properties of the victory
equilibrium and the stalemate equilibrium, we note there are two other
equilibrium points to the system defined by Eqs.~\ref{eq:model-fa}: $CR=0$, $%
SB = -\frac{\lambda_S}{f_S} $ and an equilibrium similar to Eqs.~\ref%
{eq:SG-fa}--\ref{eq:CR-fa}, but with $SB_{fi}$ the negative root of the
quadratic that produced Eq.~\ref{eq:SG-fa}. Because both of these equilibria
consist of negative values, which cannot be realized, we do not analyze them
further. We next show that the victory condition defined in section~\ref%
{ssec:model:fa} corresponds to the stability conditions for the equilibria
points.

\begin{theorem}
\label{thm-vic-con-fa} For the system of differential equations defined by
Eqs.~\ref{eq:model-directint}, the Blue victory equilibrium ($CR = SR = 0$)
is stable if and only if $\frac{r_C-A_C}{1+r_C} < S$.
\end{theorem}

\begin{proof}
We first compute the Jacobian of the right hand size of differential equation defined in Eqs.~\ref{eq:2varSB-fa}--\ref{eq:2varCR-fa}
\begin{align*}
\mathbf{J_{fi}}(SB,CR) = 
\begin{pmatrix}
f_{S}(S-2SB - A_S)-h_S CR & -h_SSB\\
-h_C CR & f_C(1-S-2CR)-h_C(SB+A_C)
\end{pmatrix}
\end{align*}
For the Blue victory equilibrium ($SB=S, CR=0$), the characteristic polynomial is 
\begin{align*}
(-f_{S}(S + A_S) - \lambda)(f_C(1-S)-h_C(S+A_C) - \lambda)=0\,.
\end{align*}
The first eigenvalue, $-f_{S}(S + A_S)$, is always negative and the second  eigenvalue, $f_C(1-S)-h_C(S+A_C)$, is negative if $\frac{r_C-A_C}{1+r_C} < S$.  Writing out this condition in terms of the direct foreign intervention parameter $\lambda_C = A_C h_C$ produces the condition: $\lambda_C>f_C(1-S) + h_C S$.
\end{proof}

Theorem~\ref{thm-vic-con-fa} provides a necessary condition for Blue victory.  We have not been able to prove the stability characteristics for the stalemate equilibrium analytically, which would provide the sufficient conditions for Blue victory. This is given in Conjecture~\ref{con-vic-con-fa} and extensive numerical experimentation makes us confident that this conjecture does hold.
\begin{conjecture}
\label{con-vic-con-fa} 
\item The stalemate equilibrium defined in Eqs.~\ref{eq:SG-fa}--\ref{eq:CR-fa} is stable if and only if $\frac{r_C-A_C}{1+r_C} > S$.
\end{conjecture}

\subsection{Opportunistic Population}

\label{sec:proofs:dynamicS} We next show that victory is the only possible
outcome of the opportunistic population model of section~\ref%
{sec:model:dynamics}.

\begin{theorem}
\label{thm-vic-con-dynamics} The following four statements hold for the
system of differential equations defined by Eqs.~\ref{eq:model-dynamics}:

\begin{itemize}
\item The equilibrium $CR = SB = 0$ and $S= \frac{1}{2}$ is never stable.

\item The Blue victory equilibrium ($CR = 0$ and $SB = S = 1$ ) is always
stable.

\item The Red victory equilibrium ($SB = S = 0$ and $CR = 1$) is always
stable.

\item The stalemate equilibrium defined in Eqs.~\ref{eq:balstale:SB}--\ref%
{eq:balstale:S} is never stable.
\end{itemize}
\end{theorem}

\begin{proof}
We first compute the Jacobian of Eqs.~\ref{eq:model-dynamics}
\begin{align*}
\mathbf{J}_{op}(SB,CR,S) = 
\begin{pmatrix}
f_{S}(S-2SB)-h_S CR & -h_SSB & f_S SB\\
-h_C CR & f_C(1-S-2CR)-h_CSB & -f_C CR \\
\alpha & \alpha & -2\alpha \\
\end{pmatrix}
\end{align*}

A solution $(SB^*,CR^*,S^*)$ to the differential equation is stable if the three eigenvalues of $\mathbf{J}_{op}(SB^*,CR^*,S^*)$ have negative real parts.  
By inspection the equilibrium with $(SB^*,CR^*,S^*) = (0,0,\frac{1}{2})$ is not stable for any parameter values. \\

\noindent {\bf Blue Victory}\\
For the equilibrium $(SB^*,CR^*,S^*) = (1,0,1)$ the characteristic polynomial is 
\begin{align*}
(-h_{C} - \lambda)(\lambda^2+\lambda(2\alpha+f_S) +\alpha f_S)=0\,.
\end{align*}
The first eigenvalue, $-h_{C}$, is always negative and the second and third eigenvalues are always negative by appealing to the quadratic formula for the second term.  \\

\noindent {\bf Red Victory}\\
For the equilibrium $(SB^*,CR^*,S^*) = (0,1,0)$ the characteristic polynomial is 
\begin{align*}
(-h_{S} - \lambda)(\lambda^2+\lambda(2\alpha+f_C) +\alpha f_C)=0\,.
\end{align*}
The first eigenvalue, $-h_{S}$, is always negative and the second and third eigenvalues are always negative by appealing to the quadratic formula for the second term.  \\

\noindent {\bf Stalemate}\\

Substituting the equilibrium points Eqs.~\ref{eq:balstale:SB}--\ref{eq:balstale:S} into the Jacobian yields
\begin{align*}
\mathbf{J}_{op}(SB^*,CR^*,S^*) = 
\begin{pmatrix}
-\frac{f_Sr_{C}}{2+r_{S}+r_{C}} & -\frac{h_Sr_{C}}{2+r_{S}+r_{C}}&  \frac{f_Sr_{C}}{2+r_{S}+r_{C}} \\
- \frac{h_Cr_{S}}{2+r_{S}+r_{C}} & -\frac{f_Cr_{S}}{2+r_{S}+r_{C}}& - \frac{f_Cr_{S}}{2+r_{S}+r_{C}}  \\
\alpha & -\alpha & -2\alpha \\
\end{pmatrix}
\end{align*}
To calculate the eigenvalues, we need to find the roots of the characteristic polynomial $$g(\lambda) = \det\left(\mathbf{J}_{op}(SB^*,CR^*,S^*) - \lambda I\right)\,.$$  
Because $\lim_{\lambda \to \infty} g(\lambda) = -\infty$, if $g(0) = \det\left(\mathbf{J}_{op}(SB^*,CR^*,S^*)\right) > 0$, 
then by the Intermediate Value Theorem there must be a positive eigenvalue and the stalemate equilibrium must be unstable. 
We will now show that this is the case for all possible parameter values.  Computing the determinant:

\begin{small}
\begin{align*}
 \det\left(\mathbf{J}_{op}(SB^*,CR^*,S^*)\right) & = \frac{\alpha\left(-2f_Sr_{C}f_Cr_{S} + h_Sr_{C}f_Cr_{S} + f_Sr_{C}h_Cr_{S} +f_Sr_{C}f_Cr_{S} + f_Sr_{C}f_Cr_{S} +2h_Sr_{C}h_Cr_{S}    \right)}{(2+r_{S}+r_{C})^2} \\
& = \frac{\alpha r_{S}r_{C}\left(h_Sf_C + f_Sh_C  +2h_Sh_C    \right) }{(2+r_{S}+r_{C})^2} >0
\end{align*}
\end{small}
Therefore there will always be a positive eigenvalue associated with the stalemate equilibrium  
defined by Eqs.~\ref{eq:balstale:SB}--\ref{eq:balstale:S}, and so it cannot be stable. 
\end{proof}

\bibliographystyle{plain}
\bibliography{insurgencies}

\end{document}